\documentclass[10pt]{amsart}
\usepackage{amsmath,latexsym}
\usepackage[psamsfonts]{amssymb}
\usepackage{times}
\usepackage[mathcal]{euscript}
\numberwithin{equation}{section}

\newcommand{\bbR}{\mathbb R}

\newcommand{\bbN}{\mathbb N}

\renewcommand{\epsilon}{\varepsilon}

\newcommand{\be}{\begin{equation}}
\newcommand{\ee}{\end{equation}}

\newcommand{\spec}{\mathrm{spec}}

\newcommand{\N}{\mathbb{N}}

\newcommand{\R}{\mathbb{R}}

\newcommand{\Z}{\mathbb{Z}}

\newcommand{\cH}{{\mathcal H}}

\newcommand{\tr}{\mathrm{tr}}

\newcommand{\abs}[1]{\left\vert#1\right\vert}

\newcommand{\la}{\lambda}

\newcommand{\de}{\delta}
\newcommand{\eps}{\varepsilon}

\newtheorem{introtheorem}{Theorem}{\bf}{\it}
\newtheorem{theorem}{Theorem}[section]
\newtheorem{lemma}[theorem]{Lemma}
\newtheorem{corollary}[theorem]{Corollary}
\newtheorem{hypothesis}[theorem]{Hypothesis}

\begin{document}
\title [On the weak and ergodic limit of the spectral shift function]
{On the weak and ergodic limit of the spectral shift function}

\author{V. Borovyk }
\address{
 Department of Mathematics, University of Arizona, PO Box 210089,
Tucson, AZ  85721–0089, USA } \email{borovykv@math.arizona.edu}

\author{K. A. Makarov}
\address{
Department of Mathematics, University of Missouri, Columbia, Missouri 63211, USA}
\email{makarovk@missouri.edu}



\begin{abstract}
We discuss convergence properties of the  spectral
shift functions associated with a pair of Schr\"odinger operators with
 Dirichlet boundary  conditions at
 the end points of a finite interval
$(0, r)$ as the  length of interval approaches infinity. 

\end{abstract}

 \maketitle


\section{Introduction}

In this note we  study the relationship between the 
spectral shift function $\xi$ associated with the pair $(H_0, H)$ of 
the half-line Dirichlet  Schr\"{o}dinger operators $$H_0=-\frac {d^2} {dx^2}\quad \text{and} \quad 
 H=-\frac {d^2} {dx^2} + V(x)\quad \text{ in } \,\,L^2(0,\infty)$$ 
and the spectral shift function $\xi^r$ associated with 
the pair $(H^r_0, H^r)$ of the corresponding 
 Schr\"odinger operators 
 on the finite interval $(0,r)$ with  the Dirichlet boundary  conditions at
 the end points of the interval.

Recall that given a pair of 
 self-adjoint operators $(H_0,H)$  in a separable Hilbert space $\cH$ 
 such that the  difference  
$                               
(H +iI)^{-1} - (H_0 +iI)^{-1}
$
is a trace class operator, 
 the spectral shift function  $\xi$ 
associated with the pair $(H_0,H)$ is uniquely determined (up to an additive constant) 
 by the trace formula
\begin{equation}                                \label{trace_formula}
\tr(f(H) - f(H_0)) = \int_{\R} \xi(\la) f'(\la)d\la,
\end{equation}
valid for a wide class of functions $f$ (\cite{BiPu},  \cite{BiYa}, \cite{Kr}, \cite{Kr2}) \cite{Ya}).  
 
In the case where both $H_0$ and $H$ are
bounded from below, the standard way to fix the undetermined constant is to
require that
\begin{equation}\label{norma}
\xi (\la) = 0\quad \text{ for }\quad\la < \inf\{\spec(H_0)\cup \spec(H)\}.
\end{equation}

Under  the short range hypothesis on the potential $V$  that \begin{equation}													\label{first_moment}
\int_0^{\infty}(1 + x) \abs{V(x)}dx < \infty,
\end{equation}
 the pairs of the half-line Schr\"odinger operators $(H_0,H)$
 and their finite-interval box-appro\-ximations $(H_0^r, H^r)$ are resolvent
 comparable 
 and the spectral shift functions $\xi$ and  $\xi^r$ 
associated with the pairs $(H_0, H)$ and $(H_0^r, H^r)$
satisfying the normalization condition \eqref{norma} are well-defined. The specifics of 
the one-dimensional case is that each of 
the spectral shift functions $\xi$ and $\xi^r$
 admits   
 a unique  
 left-continuous representative,  for which we will 
keep  the same notation. 

The main result of this note is the following theorem.

\begin{introtheorem}\label{main_theorem_0}
Assume  that a real-valued function $V$ on $(0,\infty)$ satisfies
condition \eqref{first_moment}. Denote by 
  $\xi$ and $\xi^r$  the left-continuous 
spectral shift functions associated with the pairs
$(H_0, H)$  and $(H_0^r, H^r)$
of  the Dirichlet Schr\"{o}dinger operators on the semi-axis $(0,\infty)$ and  
 on the finite interval  $(0, r)$, respectively.

\begin{itemize}
\item[]
Then  for any continuous function $g$ on $\bbR$
 with compact support  both
 the (weak) limit 
\begin{equation}                                \label{xi_weak_conv}
\lim_{r\rightarrow\infty}\int_{-\infty}^\infty \xi^r( \la)g(\la)d\la =
\int_{-\infty}^\infty\xi(\la)g(\la)d\la
\end{equation}
and the Ces\`aro limit 
\item[] 
\begin{equation}                                \label{main_equation_1}
\lim_{r\to \infty} \frac 1 r \int_0^r\xi^r(\la) dr =
\xi\big(\la\big), \quad \la \in \R\setminus (\spec_{d}(H)\cup \{0\})
\end{equation}
\end{itemize}
exist. 

If, in addition, the Schr\"odinger operator  $H=-\frac {d^2} {dx^2} + V(x)$  has  no zero energy resonance,
then  convergence \eqref{main_equation_1} takes place for $\la = 0$ as well.
\end{introtheorem}

We remark that the  weak convergence result  \eqref{xi_weak_conv} 
 has been obtained in \cite{KoSc} in the  case of arbitrary dimension 
under the assumption that  the potential $V$ 
belongs to the Birman--Solomyak class $\ell^1(L^2)$.  In contrast to the Feymann-Kac path integration approach developed in
\cite{KoSc}, our approach is based on the study of
fine properties of the eigenvalue counting function available  in the one-dimensional case. 

We also remark that the  class of potentials satisfying condition \eqref{first_moment} 
is slightly different from the Birman--Solomyak class $\ell^1(L^2)$: the Birman--Solomyak condition
allows a slower decay at infinity, while condition \eqref{first_moment} admits
a $L^1$-type singularities at finite points in contrast to the fact that   only $L^2$-type
singularities are allowed in the Birman--Solomyak class.

\section{Some general convergence results}																	\label{prelim}

In this section we will develop the necessary analytic  background for proving the
convergence results   \eqref{xi_weak_conv}  and \eqref{main_equation_1}.

Assume the following hypothesis.

\begin{hypothesis}\label{h1} Assume that  $f$ is a  Riemann integrable
function on the interval $[0, 1]$ and
$g$ is a 
continuous function on  $[0,1]$.
Suppose that  
 a sequence of real-valued measurable functions $\{f_n\}_{n=1}^\infty$ on $[0, 1]$
  converges
to  $f$ pointwise on the open interval $(0, 1)$ and that the convergence 
is uniform on every compact set 
of the semi-open interval $(0, 1]$. Assume, in addition, that  
 the sequence $\{f_n\}_{n=1}^\infty$ has a $|g|$-integrable 
majorante $F$. That is, 
\begin{equation}\label{major}
|f_n(x)|\le F(x),\quad x\in [0, 1], \,\,  n\in \bbN,
\end{equation}
   and 
\begin{equation}\label{major1}
\int_0^1F(x)|g(x)|dx<\infty.
\end{equation} 
\end{hypothesis}

Our main technical result is the following lemma.

\begin{lemma}\label{main_conv_res} Assume Hypothesis \ref{h1}.	Suppose that $\{r_n\}_{n=1}^\infty$ 
 is a sequence of non-negative numbers such that 
$$\lim_{n\to \infty}r_n=\infty.$$
Then 
\begin{equation}\label{techres}
  \lim_{n\to \infty}\int_0^1 \widetilde f_n(x)g(x)dx=\int_0^1 f(x)g(x)dx, 
  \end{equation}
   where the sequence $\{\widetilde f_n\}_{n=1}^\infty$ is given by 
$$
\widetilde f_n(x)=[r_nx + f_n(x)]-[r_nx], \quad n\in \bbN,\quad x\in [0,1],
$$
and $[\cdot]$ stands
for the integer value function. 
\end{lemma}
\begin{proof}

Without loss of generality  one may 
 assume that the function   $g$ in non-negative on $ [0,1]$.

A simple change of variables shows that 
\begin{align}
\int_0^1 & \widetilde f_n(x)g(x) dx =
\frac{1}{r_n}\int_0^{r_n}\left ( \left [t+ f_n\left ({r_n}^{-1}t\right )\right ]-\left [ t \right ]\right) g\left (r_n^{-1}t\right )dx \label{integ}
\\
&=\sum_{k=0}^{[r
_n]-1}\frac{1}{r_n}\int_k^{k+1}
\left (\left [t + f_n\left ({r_n}^{-1}t \right )\right ]-\left [ t \right ]\right )g\left ({r_n}^{-1}t\right )dt + \eps_n,
\nonumber \\&
=\sum_{k=0}^{[r
_n]-1}\frac{1}{r_n}\int_0^{1}
\left (\left [t+k + f_n\left (\frac{t+k} {r_n}\right )\right ]-\left [ t+k \right ]\right )g\left (\frac{t+k}{r_n}\right )dt + \eps_n,
\nonumber  \end{align}
where
 \begin{equation}\label{eps}\eps_n = \int_{\frac {[r_n]}{r_n}}^1\widetilde f_n(x)g(x) dx.
 \end{equation}

Taking into account that $[t+k] = [t]$ for any $t$ whenever $k$ is an
integer, and $[t] = 0$ for $t\in[0,1)$, we get
\begin{align}\label{integ}
\int_0^1 & \widetilde f_n(x)g(x) dx =\sum_{k=0}^{[r
_n]-1}\frac{1}{r_n}\int_0^{1} \left [t + f_n\left (\frac{t+k}
{r_n}\right )\right ]g\left (\frac{t+k}{r_n}\right )dt + \eps_n.
\end{align}

The uniform bound
$
|\widetilde f_n(x) |\le F(x)+1
$
(with $F$ from \eqref{major})
shows that 
\begin{equation*}
|\eps_n|\le \int_{\frac {[r_n]}{r_n}}^1(F(x)+1)g(x) dx , \quad n\in \bbN,
\end{equation*}
and therefore
\begin{equation}\label{nev}
\lim_{n\to \infty} \eps_n=0.
\end{equation}

Combining \eqref{integ} and  \eqref{nev} with the estimate
$$
\left |\int_0^1
\left [t + f_n\left (\frac{t+k}{r_n}\right )\right ]
\left (g\left (\frac{t+k}{r_n}\right )-g\left (\frac{k}{r_n}\right )\right )dt
\right |\le (\|f_n\|_\infty+1)\omega_g(r_n^{-1}),
$$
where $\omega_g(\cdot )$ stands for the modulus of continuity of
the function $g$,
one concludes that 
\begin{equation}
\lim_{n\to \infty}\int_0^1 \widetilde f_n(x)g(x)dx=\lim_{n\to
\infty}\sum_{k=0}^{[r_n]-1}\frac{g (k r_n^{-1})}{r_n} \int_0^1
\left [t + f_n\left (\frac{t+k}{r_n}\right )\right ]dt
\end{equation}
(provided that the limit in the RHS exists.) Therefore, in order to prove
assertion \eqref{techres} it suffices  to establish the equality
\begin{equation}\label{osnov}
\lim_{n\to \infty}\sum_{k=0}^{[r_n]-1}\frac{g (k r_n^{-1})}{r_n} \int_0^1
\left [t + f_n\left (\frac{t+k}{r_n}\right )\right ]dt = \int_0^1f(x)g(x) dx.
\end{equation}

To prove \eqref{osnov}, assume temporarily that the sequence of
functions $\{f_n\}_{n=1}^\infty$ converges to $f$ uniformly on the
closed interval $[0,1]$.

Since the  function $t\mapsto [t]$ is monotone,
from the inequality
$$
t+\inf_{\left [\frac{k}{r_n},\frac{k+1}{r_n}\right ]}f-\delta_n
\le t+f_n\left (\frac{t+k}{r_n}\right )\le t+
\sup_{\left [\frac{k}{r_n},\frac{k+1}{r_n}\right ]}f+\delta_n,
$$
$$
t\in [0,1], \quad
 k=0, 1 , \,...\,, [r_n]-1,
$$
where  $
\delta_n=\|f-f_n\|_\infty
$, one immediately obtains  that
\begin{equation}\label{ineqq}
\left [t+\inf_{\left [\frac{k}{r_n},\frac{k+1}{r_n}\right ]}f-\delta_n\right ]
\le \left [t+f_n\left (\frac{t+k}{r_n}\right )\right ]
\le \left [ t+
\sup_{\left [\frac{k}{r_n},\frac{k+1}{r_n}\right ]}f+\delta_n\right ],
\end{equation}
$$
t\in [0,1], \quad
 k=0, 1 , \,...\,, [r_n]-1.
$$

Integrating \eqref{ineqq} against $t$  from $0$ to $1$ and noticing that
$$
\int_0^1[ t + a]\,dt = a, \quad \text{ for all } \,\,\,a\in \bbR,
$$
and that $g$ is a non-negative function,
one obtains  the  following two-sided estimate
\begin{align}
\label{ineq} \sum_{k=0}^{[r_n]-1}\frac{g(\frac k {r_n})}{r_n}\left(\inf_{\left [\frac{k}{r_n},\frac{k+1}{r_n}\right ]}f - \delta_n\right)
\le &\sum_{k=0}^{[r_n]-1}\frac{g (\frac k {r_n})}{r_n} \int_0^1
\left [t + f_n\left (\frac{t+k}{r_n}\right )\right ]dt
\\ \notag &
\le
\sum_{k=0}^{[r_n]-1}\frac{g(\frac k {r_n})}{r_n}\left(
 \sup_{\left [\frac{k}{r_n},\frac{k+1}{r_n}\right ]}f +\delta_n \right).
\end{align}

Since $f$ is a Riemann integrable function by hypothesis and the
function $g$ is continuous on $[0,1]$, one concludes  that 
\begin{equation}\label{vto}
\lim_{n\to\infty}\frac{1}{r_n}
\sum_{k=0}^{[r_n]-1}g(kr_n^{-1})\inf_{\left [\frac{k}{r_n},\frac{k+1}{r_n}\right ]}f
=\lim_{n\to\infty}\frac{1}{r_n}
\sum_{k=0}^{[r_n]-1}g(kr_n^{-1})\sup_{\left [\frac{k}{r_n},\frac{k+1}{r_n}\right ]}f
\end{equation}
$$=\int_0^1f(x)g(x)dx.
$$
The additional assumption that the sequence
 $\{f_n\}_{n=1}^\infty$ converges to $f$ uniformly on the closed interval
 $[0,1]$ means that
$$ \lim_{n\to\infty}\delta_n=0, $$ which together with \eqref{ineq} and
\eqref{vto} proves \eqref{osnov}. This completes the proof of
\eqref{techres} (provided that $\{f_n\}_{n=1}^\infty$ converges to $f$ uniformly on the
 $[0,1]$).

To remove the extra assumption, we proceed as follows.

Given $0<\varepsilon <1$, one gets the inequality  
\begin{equation}\label{vare}
\overline{ \lim_{n\to \infty}}\left |\int_0^1 \widetilde f_n(x)g(x)dx-
\int_0^1  f(x)g(x)dx\right |
\end{equation}
$$
\le \overline{ \lim_{n\to \infty}}\left |\int_\varepsilon ^1 \widetilde f_n(x)g(x)dx-
\int_\varepsilon^1  f(x)g(x)dx\right |
$$
$$
+ \int_0^\varepsilon (2F(x)+1)g(x)dx.
$$
By hypothesis the sequence $\{f_n\}_{n=1}^\infty$ converges uniformly on 
the interval $(\varepsilon , 1]$ and therefore by the first part of the proof
one concludes that 
\begin{equation}\label{comb}
  \lim_{n\to \infty}\left |\int_\varepsilon ^1 \widetilde f_n(x)g(x)dx-
\int_\varepsilon^1  f(x)g(x)dx\right |=0.
\end{equation}

Combining \eqref{vare} and \eqref{comb} one obtains the inequality
\begin{equation}\label{malo}
\overline{ \lim_{n\to \infty}}\left |\int_0^1 \widetilde f_n(x)g(x)dx-
\int_0^1  f(x)g(x)dx\right |\le \int_0^\varepsilon (2F(x)+1)|g(x)|dx. 
\end{equation}
By the second inequality in \eqref{techres}, 
the right hand side  of \eqref{malo}
can be made arbitrary small
by an appropriate choice of $\varepsilon$ and hence 
$$
 \lim_{n\to \infty}\left |\int_0^1 \widetilde f_n(x)g(x)dx-
\int_0^1  f(x)g(x)dx\right |=0
$$
which completes the proof of the lemma.
\end{proof}

\begin{corollary}                                                                                                                   \label{ergodic_cor}
Let $h$ be a real-valued measurable bounded measurable function on $[0, \infty)$. Suppose that
the limit
$$
A=\lim_{x\to \infty}h(x)
$$
exists.
Then
\begin{equation}\label{erglim}
\lim_{r\to \infty}\frac{1}{r}\int_0^r ([x + h(x)]-[x])\,dx = A.
\end{equation}
\end{corollary}
\begin{proof}
We will prove that convergence \eqref{erglim} holds as $r$
approaches infinity along an arbitrary sequence
$\{r_n\}_{n=1}^{\infty} $ of positive numbers such that 
$
\lim_{n\to \infty}{r_n}=\infty
$
.

Introduce the sequence of
functions $\{f_n\}_{n=1}^\infty$  by
$$
f_n(x)=h(r_nx), \quad x \in [0,1].
$$ The sequence    $\{f_n\}_{n=1}^\infty$ converges pointwise to a constant function $f$ given by
$$A=\lim_{x\to \infty}h(x),
$$
and, moreover, the convergence is uniform on every compact set of
the semi-open interval $(0,1]$. Since, by hypothesis, the function
$h$ is bounded, so is the sequence $ \{f_n\}_{n=1}^\infty$ and
hence one can apply Lemma \ref{main_conv_res} to conclude that
\begin{equation}\label{momo}
  \lim_{n\to \infty} \int_0^1 ([r_nx + f_n(x)]-[r_nx])\,dx=\int_0^1A\,dx =A.
\end{equation}
By a  change of variables one  gets
$$
\frac{1}{r_n}\int_0^{r_n} ([x + h(x)]-[x])\,dx=\int_0^1 ([r_nx +
f_n(x)]-[r_nx])\,dx,
$$
which together with \eqref{momo} proves that
$$
\lim_{n\to \infty}\frac{1}{r_n}\int_0^{r_n} ([x + h(x)]-[x])\,dx = A
$$
and hence 
\eqref{erglim} holds, since 
  $\{r_n\}_{n=1}^{\infty} $ is an arbitrary sequence.
\end{proof}

\section{Proof of Theorem 1}																					\label{proof}

\begin{proof}
As it easily follows from the trace formula \eqref{trace_formula},
the  value of the spectral shift function $\xi (\la)$  on the
negative semi-axis  associated with the pair $(H_0,H)$ of the
(Dirichlet) Schr\"odinger operators is directly linked to the
number of negative eigenvalues of the operator $H$ that are
smaller than $\la$,
\begin{equation}                                                                                       \label{xi_negative}
\xi(\lambda)=-N(\lambda), \quad \lambda<0.
\end{equation}

Analogously,
 the spectral shift function $\xi^r$ associated with the finite-interval Schr\"odinger
operators $(H_0^r, H^r)$ on the negative semi-axis can  be
computed via the eigenvalue counting
 function  $N^r(\la)$  for the Schr\"odinger operator $H^r$
(cf. \cite{Kr}),
\begin{equation}\label{negspec}
\xi^r( \la) =  - N^r(\la),\quad \lambda<0.
\end{equation}

For $\lambda \ge 0$, the function $\xi$ admits the representation
in terms of the phase shift $\de$, associated with the potential
$V$ (see, e.g., \cite{BiKr}, \cite{BiYa},\cite{BuFa}, \cite{Ya}),
\begin{equation}
\label{xi_delta} \xi(\la) =  - \pi^{-1}\de(\sqrt{\la}), \quad \la
\ge 0.
\end{equation}
On the other hand, on the positive semi-axis the function $\xi^r$
can be represented as
\begin{equation} \label{xi_deltar}
\xi^r( \la) =    \left[ \pi^{-1}r\sqrt{\la} \right]
-\left[ \pi^{-1}r\sqrt{\la} +\pi^{-1}
 \de^r(
\sqrt{\la})\right], \quad \la \ge 0,\quad r >0,
\end{equation}
where $\de^r$ is the phase shift associated with the cut off potential $V^r$
given by
\begin{equation}\label{V_r_formula}
V^r(x) = \begin{cases}V(x), \quad &0\le x \le r,\\
0, \quad &x > r.
  \end{cases}\end{equation}

To prove \eqref{xi_deltar}  one observes that 
\begin{equation}                                                                               \label{xi_counting}
\xi^r( \la) = N_0^r(\lambda) - N^r(\la),\quad \lambda \ge 0,\quad
r >0,
\end{equation}
where $N_0^r$ is the eigenvalue counting function
for the Schr\"odinger operator $H^r_0$ on the finite interval.

Since 
$$
N_0^r(\lambda)=\left[ \pi^{-1}r\sqrt{\la} \right], \quad \lambda\ge 0,
$$
and 
$N^r$ can be represented as
\begin{equation}\label{pastur}
N^r(\la)=\left[ \pi^{-1}r\sqrt{\la} +\pi^{-1}
 \de^r\big(
\sqrt{\la}\big)\right], \quad \lambda \ge 0,
\end{equation}
(by the well known counting principle  that is a direct  
 consequence of the Sturm oscillation theorem
 (see, e.g., \cite[Ch. II, Sec. 6]{PaFi})), one gets \eqref{xi_deltar}.

In particular, since the counting functions  $N$, $N^r$, and $N_0^r$ are continuous from the left and
under hypothesis \eqref{first_moment} the phase shift $\de$ is continuous on $[0, \infty)$, equations
\eqref{xi_negative}, \eqref{xi_delta} and
\eqref{xi_deltar}, \eqref{xi_counting}
determine the 
left-continuous representatives for the spectral shift functions 
 $\xi$ and $\xi^r$, respectively.

To prove  the first assertion \eqref{xi_weak_conv} of the theorem we proceed as follows.

Given a continuous function $g$
with compact support, one splits the left hand side of \eqref{xi_weak_conv}
into two parts
$$
\int_{-\infty}^\infty \xi^r( \la)g(\la)d\la = I_r+II_r ,\quad r
>0,
$$
where
$$
I_r=
\int_{-\infty}^0 \xi^r( \la)g(\la)d\la
\quad \text{ and }  \quad
II_r =\int_{0}^\infty \xi^r( \la)g(\la)d\la.
$$

First we prove that
\begin{equation}\label{limI}
\lim_{r\to \infty}I_r=
\int_{-\infty}^0 \xi( \la)g(\la)d\la.
\end{equation}

Assume that the half-line   Schr\"odinger operator $H$
 has $m$, $m\ge 0$, negative eigenvalues denoted by
$$\lambda_1(H)< \lambda_2(H) < ... <\lambda_m(H).$$

Then in accordance with a result in \cite{BEWZ}, there exists a
$r_0$ such that for all $r>r_0$ the operator $H^r$ has  $m^r$, $m^r\ge m$,
negative eigenvalues
 $$\lambda_1(H^r)< \lambda_2(H^r) < ... <\lambda_{m^r}(H^r),\quad r >0,$$
and, in addition, 
\begin{equation}\label{limm}
\lim_{r\to \infty } \lambda_k(H^r)=\lambda_k(H), \quad k=1, 2, \, ..., \, m.
\end{equation}
Moreover, for any $\varepsilon >0$ there exists $r(\varepsilon)$ such that
$$
-\varepsilon < \lambda_k(H^r)<0, \quad k=m+1, ..., m^r, \quad r>r(\varepsilon).
$$

Taking into account \eqref{negspec}, from \eqref{limm} one gets that
\begin{equation}                                        \label{xi_negative_conv}
\lim_{r\rightarrow\infty}\xi^r( \la) = \xi(\la), \quad \la \in (-\infty, 0)\setminus \spec_{d}(H),
\end{equation}
where $\spec_{d}(H)$ denotes the discrete (negative) spectrum of the operator $H$.

To deduce \eqref{limI} from \eqref{xi_negative_conv} it suffices to apply a Bargmann-type estimate
 \cite{SCHM} that provides an upper bound for
 the number $m^r$ of the negative eigenvalues of the operator $H^r$ as follows
\begin{equation}\label{barg}
m^r \le \int_0^r x\,V_-(x)\,dx\le \int_0^\infty x\,V_-(x)\,dx<\infty, \quad r >0,
\end{equation}
where $V_-$ stands for  the negative part of the potential
$$
V_-(x)=\begin{cases}
-V(x), &V(x)<0\\
0, &\text{otherwise.}
\end{cases}
$$
Indeed, since
$$
|\xi^r(\lambda)|=N^r(\lambda)\le m^r\le 
\int_0^\infty x\,V_-(x)\,dx\le 
\int_0^\infty(1+ x)|V(x)|\,dx<\infty,\quad \lambda <0,$$
equality
 \eqref{limI} follows from  
 \eqref{xi_negative_conv} by  the dominated convergence theorem.

Next  we prove that
\begin{equation}\label{limII}
\lim_{r\to \infty}II_{r}=
\int^{\infty}_0 \xi( \la)g(\la)d\la.
\end{equation}

To get \eqref{limII} it suffices to show that  for every
sequence $\{r_n\}_{n=1}^\infty $ of non-negative numbers $r_n$
 such that $\lim_{n\to\infty}r_n = \infty$ the equality
 \begin{equation}\label{limIIa}
\lim_{n\to \infty}\int^{\infty}_0 \xi^{r_n}( \la)g(\la)d\la=
\int^{\infty}_0 \xi( \la)g(\la)d\la
\end{equation}
holds.

From \eqref{xi_deltar} one derives that
$$
\int_0^\infty \xi^{r}(\lambda) g(\lambda)d\lambda=
\int_0^\infty\left (\left[ \pi^{-1}r\sqrt{\la} \right]+\left[ \pi^{-1}{r}\sqrt{\la} -\pi^{-1}
 \de^{r}(
\sqrt{\la})\right]\right )g(\lambda)d\lambda,\quad r >0,
$$
 and therefore, after a change of variables, one obtains that
$$
\int_0^\infty \xi^{r_n}(\lambda) g(\lambda)d\lambda= \int_0^1\left
(\left[ {t_n}\la \right]-\left[ {t_n}\la +f_n( \la)\right]\right )
\widehat g(\lambda)d\lambda,\quad n\in\N,
$$
where
$
{t_n}= \pi^{-1}{r_n}$,  $
\widehat g(\lambda)=2\la g(\lambda^2)$,  $ \la \ge 0
$,  and  the sequence of functions $\{ f_n\}_{n=1}^\infty$ is given by
\begin{equation}\label{seq}
f_n(\lambda)=\pi^{-1}
 \de^{r_n\pi}(\lambda),\quad n\in\N.
\end{equation}

Without loss of generality (by rescaling) one may assume that
$g(\la)=0$ for $\la >1$ and hence to prove \eqref{limIIa} it
remains to check that the sequence of functions \eqref{seq}
 satisfies the hypotheses of Lemma \ref{main_conv_res} with
$$
f(\lambda)=\pi^{-1}
 \de(\lambda), \quad \la >0.
$$

Indeed,
from the phase equation (\cite{Calo}, p.~11, Eq.~(13);
\cite{BrGe})
\begin{equation}                            \label{delta_derivative_r}
\frac{d}{dr}\de^r( k) = -k^{-1} V(r) \sin^2\left(kr + \de^{r}(
k)\right) , \quad  k > 0, \quad r > 0,
\end{equation}
one easily concludes that
\begin{equation}\label{fazy}
\abs{\de^r( \sqrt{\la}) - \de(\sqrt{\la})}
 \le \frac{1}{\sqrt{\la}}\int_r^\infty \abs{ V(r')} dr', \quad \la >0,\quad r >0,
\end{equation}
 and therefore
\begin{equation}\label{vageq}
\abs{f_n( \la) - f({\la})}
 \le \frac{ 1}{\pi\la}\int_{\pi r_n}^\infty \abs{ V(r')} dr', \quad \la >0,\quad n\in\N.
\end{equation}
which proves that
\begin{equation}\label{eins}
\lim_{n\to \infty} f_n(\lambda)=f(\lambda), \quad \lambda >0,
\end{equation}
and that the convergence in \eqref{eins} takes place  uniformly
 on every compact subset of the semi-open interval $(0, 1]$.

Moreover, the following bound
\begin{equation}\label{nerav}
|f_n( \la)| \le \sup_{\lambda>0} |f({\la})|+
  \frac{1}{\pi \la}\int_{0}^\infty \abs{ V(r')} dr',\quad
  \lambda>0,\quad n\in\N,
\end{equation}
holds  and therefore the sequence \eqref{seq} has a $\widehat g$-integrable majorante.

Finally, under the short range hypothesis \eqref{first_moment} the
phase shift $\delta(\lambda)$ is a continuous function on
$(0,\infty)$ and by the Levinson Theorem,
\begin{equation}\label{lev}
\lim_{\lambda\downarrow 0}\delta(\sqrt{\la}) =\pi m, \quad m\in\Z_+,
\end{equation}
if there is no zero-energy resonance
and
$$
\lim_{\lambda\downarrow 0}\delta(\sqrt{\la}) = \pi \left
(m+\frac{1}{2}\right ), \quad m\in\Z_+, \quad \text{otherwise.}
$$
Therefore, the function $f$ is continuous on $[0,1]$ and hence it
is Riemann integrable.

Now one can apply Lemma \ref{main_conv_res} to  conclude that
\begin{align*}
\lim_{n\to \infty}\int_0^\infty \xi^{r_n\pi}(\lambda)g(\lambda)d\lambda
&=-\int_0^1 f(\lambda) \widehat  g(\lambda)d\lambda
=-\frac{1}{\pi}\int_0^1 \delta(\lambda) 2\la g(\la^2)d \la
\\
&=-\frac{1}{\pi}\int_0^1 \delta(\sqrt{\lambda})  g(\la)d \la
=\int_0^\infty \xi(\la)g(\la) d\la,
\end{align*}
which proves \eqref{limIIa}.

The proof of \eqref{xi_weak_conv} is complete.

Now, we will prove the second statement \eqref{main_equation_1} of the theorem.

 First we remark that equation~\eqref{xi_negative_conv} immediately
implies \eqref{main_equation_1} for $\la < 0,$ $\la\notin \spec_d(H)$.

In order to prove convergence~\eqref{main_equation_1} for a fixed $\la > 0$,
we again use equation~\eqref{xi_deltar} to obtain
\begin{equation*}
\frac 1 r \int_0^r \xi^r(\la)\,dr = \frac 1 r \int_0^r \left[ \pi^{-1}r\sqrt{\la} \right]
-\left[ \pi^{-1}r\sqrt{\la} +\pi^{-1}
 \de^r(
\sqrt{\la})\right]\,dr,\quad\la> 0,\quad r >0.
\end{equation*}
After the change of variables $x = \pi^{-1}r\sqrt{\la}$, we get
\begin{equation}                                                                                                                                    \label{vspomogat_eq}
\frac 1 r\int_0^r \xi^r(\la)\,dr = \frac \pi {\sqrt{\la}r}\int_0^{\frac {r\sqrt{\la}}{\pi}} \left (\left[ x\right]
-\left[ x +\pi^{-1}
 \de^{\pi x/\sqrt{\la}}(
\sqrt{\la})\right] \right )\,dx,\;\la> 0,\,r > 0.
\end{equation}
From \eqref{fazy} it follows that  the function $h$
given by
$$
h(x)=\de^{\pi x/\sqrt{\la}}( \sqrt{\la}),\quad x > 0,\quad \la >
0,
$$
is a bounded function and that
$$
\lim_{x\to \infty}h(x)=\de(\sqrt{\la}),\quad \la > 0.
$$
Hence, applying
Corollary~\ref{ergodic_cor}  one concludes that
\begin{equation}                                                                                                                                    \label{delta_conv_1}
\lim_{r\to \infty}\frac{\pi}{\sqrt{\la}r}\int_0^{\frac {r\sqrt{\la}}{\pi}} \left(\left[ x \right]
-\left[ x+\pi^{-1}
 \de^{\pi x/\sqrt{\la}}(
\sqrt{\la})\right]\right)\,dx = -\de(\sqrt{\la}),\; \la > 0,
\end{equation}
and, therefore,
\begin{equation*}
\lim_{r\to \infty}\frac 1 r\int_0^r \xi^r(\la)\,dr = -\pi^{-1}\de(\sqrt{\la}),\quad \la > 0,
\end{equation*}
which together with \eqref{xi_delta} proves  \eqref{main_equation_1} for $\la > 0$.

Finally,  to prove \eqref{main_equation_1} for $\la=0$ in the case
of the absence of a zero energy resonance, notice that  the
corresponding Jost function $\mathcal{F}$ does not vanish on the
closed semi-interval $[0, \infty)$. Since
\begin{equation*}
\lim_{r \rightarrow \infty} \mathcal{F}^r(k) = \mathcal{F}(k), \quad k \ge 0,
\end{equation*}
it is obvious that $\mathcal{F}^r(k)$ is not zero for $k\in[0, \infty)$ for all $r$ large enough and hence
\begin{equation}                                                                                                                    \label{delta_conv_zero}
\lim_{r \rightarrow \infty} \de(r, 0) = \de(0)
\end{equation}
for $\de(r, k) = \arg(\mathcal{F}^r(k))$ and $\de(k) = \arg(\mathcal{F}(k))$.

Now equations \eqref{vspomogat_eq}, \eqref{delta_conv_zero}, and Corollary~\ref{ergodic_cor}
imply
\begin{equation*}
\lim_{r\to \infty}\frac 1 r\int_0^r \xi^r(0)\,dr = -\pi^{-1}\de(0).
\end{equation*}
By the Levinson Theorem (see  \eqref{lev}), the quantity $ \pi^{-1}\de(0)$ coincides with the number $N(0)$ of
(negative) eigenvalues of the Schr\"odinger operator $H$. Since
$$
\xi(0)=\lim_{\eps\downarrow 0}\xi (-\eps)=-N(0),
$$
one concludes that
$$
\lim_{r\to \infty}\frac 1 r\int_0^r \xi^r(0)\,dr =\xi(0)
$$
which completes the proof of
Theorem~\ref{main_theorem_0}.
\end{proof}

{\bf Acknowledgments} 
We would like to thank  Fritz Gesztesy and Vadim Kostrykin for useful discussions on this topic.

\end{document}